\documentclass[11pt, a4paper]{amsart}
\usepackage[latin1]{inputenc}
\usepackage[english]{babel}
\usepackage{amssymb, amsmath, amsthm}
\usepackage{hyperref}
\usepackage{geometry}
\usepackage{fancyhdr}
\usepackage{relsize}
\usepackage{upgreek}
\usepackage{verbatim}
\usepackage{enumitem}
\usepackage{calrsfs}
\usepackage{todonotes}
\usepackage{blkarray}
\usepackage{array}
\usepackage{tikz}
\usetikzlibrary{matrix,arrows,calc,cd}
\usepackage{package}
\newcommand{\sv}[3]{\begin{smallmatrix}#1\\#2\\#3\end{smallmatrix}}

\newcommand{\vd}{\smash{\vdots}}
\newcolumntype{C}{>{\hfil$}p{.5em}<{$\hfil}}
\newcommand{\ray}[2]{\smash{\begin{array}{CC} #1 & #2 \end{array}}}

\DeclareMathOperator{\act}{act}
\DeclareMathOperator{\Att}{Att}

\newcommand{\g}{\mathfrak{g}}
\newcommand{\gl}{\mathfrak{gl}}
\title{Weight Spaces and Attracting Sets for Torus Actions on Quiver Moduli}
\author{Magdalena Boos \and Hans Franzen}
\address{
Hans Franzen\\ Ruhr University Bochum\\ Faculty of Mathematics\\
Universit\"ats\-strasse 150\\
44780 Bochum (Germany)}
\email{hans.franzen@rub.de}
\address{
Magdalena Boos\\ Ruhr University Bochum\\ Faculty of Mathematics\\
Universit\"ats\-strasse 150\\
44780 Bochum (Germany)}
\email{magdalena.boos-math@rub.de}
\date{}
\begin{document}
		
	\begin{abstract}
		We study torus actions on moduli spaces of quivers. First we give a description of the weight spaces of the induced action of the tangent space to a torus-fixed point. Then we focus on actions of tori of rank one and derive an explicit form for the attractors in the Bia{\l}ynicki-Birula decomposition.
	\end{abstract}

	\maketitle
	
	\section{Introduction}
	
	An initial motivation for the study of moduli spaces of quivers is the classification problem for quiver representations. There is the well-known trichotomy for representations of quivers over an algebraically closed field which separates quivers by their representation type. They are either of finite type, tame, or wild. For representations of finite type or tame quivers it is possible to parametrize all indecomposable representations up to isomorphism but for any wild quiver -- i.e.\ a quiver which is not an orientation of a Dynkin digram of type $A$, $D$, or $E$, or of an extended Dynkin diagram of type $\tilde{A}$, $\tilde{D}$, or $\tilde{E}$ -- this is a very difficult problem. Kac uses in \cite{Kac:80, Kac:82, Kac:83} geometric methods to understand the structure of the entirety of indecomposable representations of a wild quiver. He counts isomorphism classes of absolutely indecomposable representations over finite fields and studies the counting functions which he shows are polynomials over the integers. He formulates a series of conjectures related to these counting polynomials, one of which states that the space of indecomposable representations possesses a cell decomposition.
	
	King applies in \cite{King:94} methods of Mumford's geometric invariant theory to construct moduli spaces of stable representations of a quiver. Stable representations are indecomposable, so these moduli parametrize an interesting subclass of indecomposable representations up to isomorphism. If the quiver does not contain oriented cycles and if semi-stability and stability agree then these moduli spaces are smooth projective varieties. They have very nice geometric properties. They are rational \cite{Schofield:01}, their cohomology is concentrated in even degrees and the cycle map is an isomorphism \cite{KW:95}. This suggests that (smooth projective) quiver moduli might possess a cell decomposition. It is, so far, unknown if they do.
	
	One possible approach to this problem is to find reasonable actions of algebraic tori on these moduli spaces. Such actions were introduced by Weist in \cite{Weist:13}. It turns out that the fixed point components under these actions can again be identified with moduli spaces of another quiver, the so-called universal abelian covering quiver. The components are not always isolated fixed points so this method does in general not provide a cell decomposition of the moduli space. However, interesting information about cohomology can be extracted from these actions. If we furthermore restrict to actions of tori of rank one, then we may apply methods of Bia{\l}ynicki-Birula \cite{BB:73} to obtain a partition of the space into affine bundles over fixed point components. The dimensions of the fibers of these affine spaces agree with dimensions of weight spaces of the induced action on the tangent space.
	
	In this paper we obtain two main results. We obtain a representation-theoretic description of the weight spaces of the tangent space to a torus fixed point. More precisely, we identify the weight space with an $\Ext$-group of the universal abelian covering quiver. We moreover derive an explicit formula for its dimension in terms of the Euler form of the universal abelian covering quiver. This is Theorem \ref{t:wtsp}. Our second main result, Theorem \ref{t:att_explicit}, is concerned with the description of the attractors of the Bia{\l}ynicki-Birula decomposition for actions of tori of rank one. We give a constructive way to compute a full system of representatives of the attractors by the choice of vector space complements of certain, explicitly given, linear maps.
	The attracting sets are also studied in a recent paper by Kinser and Weist \cite{KW:19}. We refer to section \ref{s:attr} for a discussion of the relation of their results to ours.

	We structure the article as follows. In section \ref{s:quivmod}, basics and references for quiver moduli are given; results on torus actions on the latter can be found in section \ref{s:torusActions}. We state the Bia{\l}ynicki-Birula decomposition of a smooth projective $\G_m$-variety in section \ref{s:bialynicki} and deduce consequences for its Poincar\'{e} polynomial. In section \ref{s:torusquot}, torus actions on geometric quotients are considered from a general point of view and induced torus actions on the tangent spaces are introduced. These results are applied in section \ref{s:wtsp} to the more concrete setup of quiver moduli. We calculate a dimension formula for the  weight spaces of tangent spaces on stable moduli spaces. Attractors are calculated in section \ref{s:attr} before applying our results to the example of the $l+1$-Kronecker quiver in section \ref{s:appl}.

	\begin{ack*}
		These results were obtained while the authors spent 2 months as Leibniz Fellows at Mathematisches Forschungsinstitut Oberwolfach. We would like to thank the Leibniz Foundation and the MFO for support and hospitality. We would like to thank M.\ Brion, R.\ Kinser, and T.\ Weist for interesting discussions and helpful remarks.
	\end{ack*}

	\section{Quiver Moduli}\label{s:quivmod}
	
	Let $Q$ be a quiver. This is, by definition, an oriented graph. Its sets of vertices and arrows are denoted $Q_0$ and $Q_1$, respectively. They need not be finite. For an arrow $a \in Q_1$ denote $s(a)$ and $t(a)$ its source and target. All quivers will be assumed to have only finitely many arrows $a \in Q_1$ such that $s(a) = i$ and $t(a) = j$ for any two vertices $i,j \in Q_0$. We assume that the reader is familiar with the basics of representations of quivers and refer to \cite[II.1]{ASS:06} for more details. Let $M$ be a representation of $Q$ over a field $k$. All representations in the article are assumed to be finite-dimensional, i.e.\ $\dim \smash{\bigoplus_{i \in Q_0}} M_i < \infty$. Let $\rep_k(Q)$ be the cateogory of finite-dimensional representations of $Q$ over $k$. Let $\Lambda(Q)$ be the subgroup of those $d \in \smash{\Z^{Q_0}}$ for which $d_i = 0$ for all but finitely many $i \in Q_0$. Let $\Lambda_+(Q) = \Lambda(Q) \cap \smash{\Z_{\geq 0}^{Q_0}}$. The elements of $\Lambda_+(Q)$ will be called dimension vectors of $Q$. The dimension vector of $M$ is the tuple $\dimvect M := (\dim M_i)_{i \in Q_0} \in \Lambda_+(Q)$. Given $d \in \Lambda_+(Q)$ we fix vector spaces $V_i$ of dimension $d_i$ and we define the $k$-vector space
	$$
		R(Q,d) := \bigoplus_{a \in Q_1} \Hom(V_{s(a)},V_{t(a)}).
	$$
	Elements of this vector space give rise to representations of $Q$ of dimension vector $d$. The space $R(Q,d)$ is acted upon by the group $G_d := \prod_{i \in Q_0} \GL(V_i)$ via
	$$
		g \cdot M = (g_{t(a)}M_ag_{s(a)}^{-1})_{a \in Q_0}
	$$
	where $g = (g_i)_{i \in Q_0} \in G_d$ and $M = (M_a)_{a \in Q_1} \in R(Q,d)$. Two elements of $R(Q,d)$ are isomorphic as representations if and only if they lie in the same $G_d$-orbit. Note that the image $\Delta$ of the homomorphism $k^\times \to G_d$, which embeds $t \in k^\times$ diagonally into each component, acts trivially on $R(Q,d)$. The action of $G_d$ hence descends to an action of $PG_d := G_d/\Delta$.
	
	Suppose in the following that the field $k$ is algebraically closed. Then $R(Q,d)$ is the set of closed points of an affine space over $k$ and $G_d$ and $PG_d$ are sets of closed points of reductive algebraic groups. By abuse of notation, we denote this affine space also by $R(Q,d)$ and the algebraic groups also by $G_d$ and $PG_d$.
	
	Let $\theta : \Lambda(Q) \to \Z$ be a $\Z$-linear map. We call it a stability condition. It is given by integers $\theta_i$ for $i \in Q_0$. For a dimension vector $d \in \Lambda_+(Q) - \{0\}$ we define the $\theta$-slope $\mu(d) := \theta(d)/(\sum_i d_i)$. For a representation $M \neq 0$ of $Q$ let $\mu(M) := \mu(\dimvect M)$. A representation $M$ is called $\theta$-stable (resp.\ semi-stable) if $\mu(M') < \mu(M)$ (resp. $\mu(M') \leq \mu(M)$) for every proper non-zero subrepresentation $M'$ of $M$. We obtain two Zariski open, but possibly empty, subsets
	$$
		R(Q,d)^{\theta-\st} \sub R(Q,d)^{\theta-\sst} \sub R(Q,d)
	$$
	whose closed points are precisely the sets of those closed points of $R(Q,d)$ which are $\theta$-stable (resp.\ $\theta$-semi-stable) as representations of $Q$. In \cite[Prop.\ 3.1]{King:94} King shows that $R(Q,d)^{\theta-\sst}$ agrees with the set of semi-stable points with respect to the $PG_d$-linearization of the (trivial) line bundle on $R(Q,d)$ which is given by the character $\chi: PG_d \to \G_m$ defined by
	$$
		\chi(g) = \prod_{i \in Q_0} (\det g_i)^{\theta(d) - \theta_i\sum_j d_j};
	$$
	observe that this is indeed a character of $PG_d$.
	Note that King requires $\theta(d) = 0$. The modification of the exponent in the definition of $\chi$ is to get rid of this assumption; it was introduced in \cite[3.4]{Reineke:08}. The set of $\theta$-stable points agrees with the set of properly stable points with respect to $\chi$ in the sense of Mumford \cite[Def.\ 1.8]{GIT:94}.
	
	We define $M^{\theta-\sst}(Q,d) := R(Q,d)^{\theta-\sst}/\!\!/PG_d$ and $M^{\theta-\st}(Q,d) := R(Q,d)^{\theta-\st}/PG_d$. The morphism $M^{\theta-\sst}(Q,d) \to M^{0-\sst}(Q,d) = \Spec(k[R(Q,d)]^{PG_d})$ is projective. Le Bruyn and Procesi show in \cite[Thm.\ 1]{LP:90} that $k[R(Q,d)]^{PG_d}$ is generated by traces along oriented cycles of $Q$. So if $Q$ is acyclic, then $M^{\theta-\sst}(Q,d)$ is a projective variety. The variety $M^{\theta-\st}(Q,d)$ is an open subset of $M^{\theta-\sst}(Q,d)$. The quotient morphism $R(Q,d)^{\theta-\st} \to M^{\theta-\st}(Q,d)$ is a $PG_d$-fiber bundle in the \'{e}tale topology with a smooth total space, so $M^{\theta-\st}(Q,d)$ must itself be smooth.
	
	Suppose that the dimension vector $d$ is $\theta$-coprime, by which we mean that $\mu(d') \neq \mu(d)$ for every $d' \in \Lambda_+(Q) - \{0,d\}$ for which $d-d' \in \Lambda_+(Q)$. It can easily be seen that this forces $d$ to be indivisible, i.e.\ the greatest common divisor of the $d_i$'s is one. Conversely, for any indivisible dimension vector $d$, there exists a stability condition $\theta$ for which $d$ is $\theta$-coprime. Under this assumption, we obtain $R(Q,d)^{\theta-\sst} = R(Q,d)^{\theta-\st}$ and, of course, also $M^{\theta-\sst}(Q,d) = M^{\theta-\st}(Q,d)$. We write just $R(Q,d)^\theta$ and $M^\theta(Q,d)$ in that case. If moreover, the quiver $Q$ is acyclic, then the considerations in the previous paragraph show that $M^\theta(Q,d)$ is a smooth projective variety.

	\section{Torus Actions on Quiver Moduli} \label{s:torusActions}
	
	Let $Q$ be a quiver, $d$ a dimension vector, and $\theta$ a stability condition. Let $T = \G_m^n$ be an algebraic torus. Suppose $T$ acts on $R(Q,d)$ by
	$$
		t.M = (w_a(t)M_a)_a
	$$
	where $w_a$ is a character of $T$. Explicitly, it is given by $w_a(t) = t_1^{w_{a,1}}\ldots t_n^{w_{a,n}}$ for integers $w_{a,\nu}$. The above action commutes with the $PG_d$-action on $R(Q,d)$ and respects $\theta$-stability. Hence it descends to an action on the quotient $M^{\theta-\st}(Q,d)$. We say that $T$ acts on $M^{\theta-\st}(Q,d)$ by the weights $w_a$. The weights $w_a$ give rise to a homomorphism $w: T \to \smash{\G_m^{Q_1}}$ of tori.	
	
	Let us briefly recall some arguments of \cite{Weist:13}. Weist shows there that a fixed point $[M] \in M^{\theta-\st}(Q,d)^T$ yields a homomorphism $\rho: T \to PG_d$; it admits a lift $\dot{\rho}: T \to G_d$. Via this lift, $T$ acts on $V_i$ by $t.v = \dot{\rho}_i(t)v$. This gives a weight space decomposition $V_i = \sum_\chi V_{i,\chi}$ which satisfies
	$$
		M_a(V_{i,\chi}) \sub V_{j,\chi+w_a}
	$$
	for all $a: i \to j$.
	
	This can be rephrased representation-theoretically. We define the (infinite) quiver $Q(w)$ as follows. We set
	\begin{align*}
		Q(w)_0 &= Q_0 \times \Z^n & Q(w)_1 &= Q_1 \times \Z^n
	\end{align*}
	and the sources and targets of the arrows are defined by $s(a,\chi) = (s(a),\chi)$ and $t(a,\chi) = (t(a),\chi+w_a)$.	
	There is a morphism $c: Q(w) \to Q$ of quivers which sends $(i,\chi)$ to $i$ and $(a,\chi)$ to $a$. Consider $\Lambda(Q(w)) \sub \Z^{Q(w)_0}$. Then $c$ extends by linearity to $c: \Lambda(Q(w)) \to \Lambda(Q)$. A dimension vector $\beta \in \Lambda_+(Q(w))$ is called compatible with $d$ if $c(\beta) = d$. Moreover, we obtain an additive functor $c: \rep_k(Q(w)) \to \rep_k(Q)$ which assigns to a representation $N$ of $Q(w)$ the representation $M = c(N)$ given by
	\begin{align*}
		M_i &= \bigoplus_{\chi \in \Z^n} N_{i,\chi} & M_a &= \sum_{\chi \in \Z^n} N_{a,\chi}.
	\end{align*}
	Clearly, $\dimvect c(N) = c(\dimvect N)$. A representation $N$ of $Q(w)$ such that $c(N) = M$ is called a lift of $M$ to $Q(w)$.
	
	We have an action of the abelian group $\Z^n$ on $Q_0 \times \Z^n$ by $\chi.(i,\xi) = (i,\chi+\xi)$. This action extends to a $\Z^n$-action on $\Lambda_+(Q(w))$. We denote the application of $\chi$ to $\beta \in \Lambda_+(Q(w))$ by $s_{\chi}(\beta)$, concretely 
	$$
		s_\chi(\beta)_{i,\xi} = \beta_{i,\chi+\xi}.
	$$
	In the same vein, we define for a representation $N$ of $Q(w)$ a representation $s_\chi(N)$ by $s_\chi(N)_{i,\xi} = N_{i,\chi+\xi}$. The map $c: \Lambda_+(Q(w)) \to \Lambda_+(Q)$ is $\Z^n$-invariant. Two elements $\beta,\gamma \in \Lambda_+(Q(w))$ are called equivalent if they lie in the same $\Z^n$-orbit. 
	The functor $c: \rep_k(Q(w)) \to \rep_k(Q)$ is also $\Z^n$-invariant. Two representations $N$ and $N'$ of $Q(w)$ are called equivalent if $N' \cong s_\chi(N)$ for some $\chi \in \Z^n$.
	
	Given a fixed point $[M] \in M^{\theta-\st}(Q,d)^T$, the choice of a lift $\dot{\rho}$ of $\rho$ determines a lift $N$ of $M$ to $Q(w)$; different lifts of $\rho$ give rise to equivalent lifts of $M$. A lift $N$ of $M$ is stable with respect to the stability condition $\smash{\hat{\theta}}$ of $Q(w)$, which is given by $\smash{\hat{\theta}_{i,\chi}} = \theta_i$ for all $\chi \in \Z^n$.
	Weist shows the following characterization of the fixed point set of the $T$-action.
	
	\begin{thm}[see {\cite[Thm.\ 3.8]{Weist:13}}] \label{t:thorsten}
		The fixed point set $M^{\theta-\st}(Q,d)^T$ is the finite disjoint union $\bigsqcup_\beta F_\beta$ into irreducible components, where $\beta$ ranges over a full system of representatives of equivalence classes of dimension vectors of $Q(w)$ that are compatible with $d$. The connected component $F_\beta$ is isomorphic to the moduli space 
		$$
			F_\beta \cong M^{\hat{\theta}-\st}(Q(w),\beta).
		$$
	\end{thm}

	\begin{rem} \label{r:actions}
		There are two important special cases which the theory is most frequently applied to.
		\begin{enumerate}
			\item In \cite{Weist:13}, Weist uses the action of $T = \G_m^{Q_1}$ which acts as $t.M = (t_aM_a)$. In our language this amounts to choosing the characters $w_a = e_a$ (the unit vector in $\smash{\Z^{Q_1}}$). In fact the whole theory can be recovered from this special case because the weights $w_a$ give a homomorphism to $\smash{\G_m^{Q_1}}$. However, for uniformity of exposition we chose to state the theory for arbitrary tori.
			\item If $T = \G_m$ is a torus of rank one then the weights $w_a$ are just integers. In this situation we can apply Bia{\l}ynicki-Birula's theory. 
		\end{enumerate}
	\end{rem}
	
	We are going to recall some of Bia{\l}ynicki-Birula's results in the following section and provide a general way to describe weight spaces of torus actions on geometric quotients. We then apply this description to the situation of item (2) of Remark \ref{r:actions} to derive a formula for the dimensions of the weight spaces.

	\section{Bia{\l}ynicki-Birula Decompositions}\label{s:bialynicki}
	
	Let $X$ be a variety, again over an algebraically closed field $k$, equipped with an action of $T = \G_m^n$. The locus $X^T$ of fixed points is a closed subset; let $\smash{\bigsqcup_{\beta \in B}} \smash{F_\beta}$ be the decomposition into connected components -- indexed by a finite set $B$. If $X$ is smooth, then so is $X^T$ (see \cite{Iversen:72} or \cite{Fogarty:73}) and hence the connected components of $X^T$ are irreducible and smooth. For a closed point $x \in F_\beta \sub X^T$, we obtain an action of $T$ on the tangent space $T_xX$ by differentiating the action map. Hence $T_xX$ possesses a weight space decomposition $T_xX = \bigoplus_{\chi \in \Z^n} (T_xX)_\chi$. The invariant part $(T_xX)_0$ is isomorphic to $T_xF_\beta$.
	
	In \cite{BB:73}, Bia{\l}ynicki-Birula describes how a smooth projective variety $X$ decomposes when acted upon by a rank one torus $T = \G_m$. For a closed point $x \in X$ consider the morphism $\G_m \to X$ which is given by the action on $x$. As $X$ is projective, this morphism extends uniquely to a morphism $\P^1 \to X$ (here we identify $\G_m = \P^1-\{0,\infty\}$). We denote the value at $0$ by $\lim_{t\to 0} t.x$ and the value at $\infty$ by $\lim_{t\to \infty} t.x$. Note that both limit points are automatically  $\G_m$-fixed points. 
	Let $F \sub X^{\G_m}$ be a connected component. Define
	\begin{align*}
		\Att^+(F) &:= \{ x\in X \mid \lim_{t\to 0} t.x \in F \} \\
		\Att^-(F) &:= \{ x\in X \mid \lim_{t\to \infty} t.x \in F \}
	\end{align*}
	They are called the plus- and minus-attractors of $F$.
	Define the maps $\pi_F^{\pm}: \Att^{\pm}(F) \to F$ which send $x$ to $\pi_F^+(x) = \lim_{t\to 0} t.x$ and $\pi_F^-(x) = \lim_{t\to \infty} t.x$, respectively. Define for a fixed point $x \in F \sub X^{\G_m}$ the attractors of $x$ as $\Att^{\pm}(x) := (\pi_F^{\pm})^{-1}(x)$.
	
	\begin{thm}[{\cite[II,Thm.\ 4.2]{BCM:02}}] \label{t:bb}
		Let $X$ be a smooth projective $\G_m$-variety. Let $X^{\G_m} = \smash{\bigsqcup_{\beta \in B}} F_\beta$ be the decomposition into irreducible/connected components. For every index $\beta \in B$ let $\smash{X_\beta^{\pm}} := \smash{\Att^{\pm}(F_\beta)}$ and $\smash{\pi_\beta^{\pm}} = \smash{\pi_{F_\beta}^{\pm}}$.
		\begin{enumerate}
			\item Each $\smash{X_\beta^{\pm}}$ is locally closed, irreducible and smooth. The map $\smash{\pi_{\beta}^{\pm}}$ is a $\G_m$-invariant morphism.
			\item The union $\bigcup_{\beta \in B} X_\beta^+$ is disjoint and equals $X$. The same holds for $\bigcup_{\beta \in B} X_\beta^-$.
			\item There exists a total order $B = \{\beta_1,\ldots,\beta_N\}$ such that $\smash{X_{\beta_1}^+} \cup \ldots \cup \smash{X_{\beta_n}^+}$ is closed in $X$ for every $n=1,\ldots,N$. For this total order $\smash{X_{\beta_N}^-} \cup \ldots \cup \smash{X_{\beta_n}^-}$ is closed for $n=1,\ldots,N$.
			\item The morphism $\smash{\pi_\beta^+}: \smash{X_\beta^+} \to F_\beta$ is a Zariski locally trivial fibration whose fiber in the point $x \in F_\beta$ is an affine space of dimension $\sum_{m>0} \dim (T_xX)_m$. An analogous statement holds for $\smash{\pi_\beta^-}$.
			% Maybe even the following is true in our setup, but I don't know.
			%\item There exist $\G_m$-equivariant isomorphisms $\smash{X_\beta^{\pm}} \cong \smash{(TX|_{F_\beta})_{\pm}}$ over $F_\beta$ where $(TX|_{F_\beta})_+$ is the subbundle of $TX|_{F_\beta}$ whose fiber in $x \in F_\beta$ is $\bigoplus_{m > 0} (T_xX)_m$ and $(TX|_{F_\beta})_-$ is the subbundle with fiber $\bigoplus_{m < 0} (T_xX)_m$.
		\end{enumerate}
	\end{thm}

	As a consequence, we obtain a description of the Poincar\'{e} polynomial of $X$, when $X$ is a smooth projective complex variety of complex dimension $d$. The Poincar\'{e} polynomial of $X$ is defined as the polynomial $P_X(t) = \sum_{i=0}^{2d} (-1)^i \dim H^i(X;\Q) t^i$ whose coefficients are the signed dimensions of the Betti cohomology groups of $X$. By the above theorem, 
	$$
		P_X(t) = \sum_{\beta \in B} t^{2d_\beta} P_{F_\beta}(t), 
	$$
	where $d_\beta = \dim_\C \bigoplus_{m > 0} (T_xX)_m$ for any $x \in F_\beta$.
	
	To apply Theorem \ref{t:bb} to actions of a higher rank torus $T = \G_m^n$, it is possible to reduce it to the rank 1 case using a suitable one-parameter subgroup. Note that the fixed point locus decomposes into finitely many components $F_\beta$. For each two points $x,y$ in the same irreducible component $F_\beta$, the weight space decomposition of the tangent spaces $T_xX$ and $T_yX$ is identical.
	% Proof: Note that the tangent bundle $TX$ is a $T$-equivariant vector bundle. Let $U$ and $U'$ be open neighborhoods of $x$ and $y$ over which $TX$ trivializes. Then $TX|_{U \cap F_\beta} \cong (U \cap F_\beta) \times T_xX$ as $T$-equivariant bundles and the same holds for $y$. But as $F_\beta$ is irreducible, $F_\beta$ intersects $U \cap U'$, so $T_xX$ and $T_yX$ must be isomorphic as $T$-representations.
	Let $C$ be the set of all characters $\chi \in \Z^n$ for which there exists $x \in X^T$ such that $(T_xX)_\chi \neq 0$. By the above arguments, $C$ is finite. If we choose a one-parameter subgroup $\lambda: \G_m \to T$ such that $\langle \lambda, \chi \rangle \neq 0$ for every $\chi \in C$ then we obtain $X^T = X^{\lambda(\G_m)}$ and also
	$$
		(T_xX)_+ = \bigoplus_{\langle \lambda,\chi \rangle \geq 0} (T_xX)_\chi.
	$$
	We also define $\Att_\lambda^+(y)$ and $\Att_\lambda^+(F)$ to be the attractor of a point $y \in X^T = X^{\lambda(\G_m)}$ resp.\ of  a closed subset $F \sub X^T$ with respect to the $\G_m$-action induced by $\lambda$. 
	
	Although pretty straight-forward, let us show how such a one-parameter subgroup $\lambda$ can be chosen in a convenient way. Consider again the set $C$ of characters which occur as a weight. Take the maximum norm $\|\chi\| = \max \{ |\chi_i| \mid i=1,\ldots,n\}$ on characters $\chi \in \Z^n$ and define $R := \max \{ \|\chi\| \mid \chi \in C \}$. Define $\lambda:\G_m \to T$ by $\lambda(z) = (z^{\lambda_1},\ldots,z^{\lambda_n})$ where $\lambda_i = (R+1)^{n-i}$. Then an easy verification shows
	
	\begin{lem} \label{l:1psg}
		\begin{enumerate}
			\item We get $\langle \lambda, \chi \rangle \neq 0$ for every $\chi \in C$; as a consequence $X^T = X^{\lambda(\G_m)}$.
			\item We have $\langle \lambda,\chi \rangle > 0$ if and only if $\chi \neq 0$ and $\chi_j > 0$, where $j = \min \{ i \in \{1,\ldots,n\} \mid \chi_i \neq 0 \}$.
		\end{enumerate}
	\end{lem}

	\section{Torus Actions on Quotients}\label{s:torusquot}
	
	By Weist's Theorem \ref{t:thorsten}, we have an explicit description of the fixed point loci. To make Theorem \ref{t:bb} explicit in this context, the only missing ingredient is a closed formula for the dimensions of the weight spaces. In order to derive this, we first look at the following general setup.
	
	Suppose we are given a variety $Y$ over an algebraically closed field $k$ and actions of a reductive algebraic group $G$ and a torus $T$ which commute. Write $gy$ for the action of $g \in G$ on $y \in Y$ and $t.y$ for $t \in T$ acting on $y$. Assume there exists a principal $G$-fiber bundle $\pi: Y \to Y/G =: X$; in particular, the $G$-action on $Y$ is free. Then, by commutativity of the actions, the torus $T$ acts on the quotient $Y/G$, in such a way that the quotient map is $T$-equivariant.
	
	Let $x = \pi(y)$ be a $T$-fixed point of $X = Y/G$. That means there exists a homomorphism $\rho: T \to G$ such that $t.y = \rho(t)y$ for all $t \in T$. This induces a natural $T$-action on the tangent space $T_{\pi(y)}(Y/G)$ as follows. First of all, we observe that the tangent space $T_{\pi(y)}(Y/G)$ is naturally identified with
	$$
		T_yY/\im( d_e \act_G(\blank,y) ),
	$$
	where $\act_G: G \times Y \to Y$ is the action map and thus $d_e \act_G(\blank,y): \mathfrak{g} \to T_yY$; this is an injective linear map because the $G$-action is free. We have for every $t \in T$ a commutative diagram
	\begin{center}
		\begin{tikzpicture}[description/.style={fill=white,inner sep=2pt}]
			\matrix(m)[matrix of math nodes, row sep=1.5em, column sep=10em, text height=1.5ex, text depth=0.25ex]
			{
				T_yY		&  T_{t.y}Y	& T_yY\\
				\mathfrak{g}	& \mathfrak{g}	& \mathfrak{g} \\
			};
			\path[->, font=\scriptsize]
			(m-1-1) edge node[auto] {$d_y\act_T(t,\blank)$} 	(m-1-2)
			(m-1-2) edge node[auto] {$d_{t.y}\act_G(\rho(t)^{-1},\blank)$} 	(m-1-3)
			(m-2-1) edge node[auto] {$d_e\act_G(\blank,y)$}		(m-1-1)
			(m-2-2) edge node[auto] {$d_e\act_G(\blank,t.y)$}		(m-1-2)
			(m-2-3) edge node[auto] {$d_e\act_G(\blank,y)$}		(m-1-3)
			(m-2-1) edge node[auto] {$\id$}		(m-2-2)
			(m-2-2) edge node[auto] {$\mathrm{Ad}_{\rho(t)^{-1}}$}		(m-2-3)
			;
		\end{tikzpicture}
	\end{center}
	so in total, $T$ acts on $T_yY$ via $d_y(y' \mapsto \rho(t)^{-1}t.y')$ --- observe that this is indeed a group action --- and the restriction to the image of $d_e\act_G(\blank,y): \mathfrak{g} \into T_yY$ is just $\mathrm{Ad}_{\rho(t)^{-1}}$. The induced action on the quotient is the natural action on $T_{\pi(y)}(Y/G)$.
	
	We turn to an important special case. Consider commuting linear actions of $G$ and $T$ on a finite-dimensional vector space $V$ and let $Y \sub V$ be an open, $G \times T$-invariant subset. In this case $T_yY \cong V$ and via this identification, we have
	\begin{align*}
		d_y\act_T(t,\blank)&: V \to V,\ v \mapsto t.v \\
		d_y\act_G(g,\blank)&: V \to V,\ v \mapsto gv.
	\end{align*}
	The commutative diagram which describes the $T$-action on $T_{\pi(y)}(Y/G) \cong V/\im(d_e\act_G(\blank,y))$ is thus the following:
	\begin{center}
		\begin{tikzpicture}[description/.style={fill=white,inner sep=2pt}]
			\matrix(m)[matrix of math nodes, row sep=1.5em, column sep=10em, text height=1.5ex, text depth=0.25ex]
			{
				V		&  V	\\
				\mathfrak{g}	& \mathfrak{g} \\
			};
			\path[->, font=\scriptsize]
			(m-1-1) edge node[auto] {$v \mapsto \rho(t)^{-1}t.v$} 	(m-1-2)
			(m-2-1) edge node[auto] {$d_e\act_G(\blank,y)$}	(m-1-1)
			(m-2-2) edge node[right] {$d_e\act_G(\blank,y)$}	(m-1-2)
			(m-2-1) edge node[auto] {$\mathrm{Ad}_{\rho(t)^{-1}}$}	(m-2-2)
			;
		\end{tikzpicture}
	\end{center}
	
	Let $\chi$ be a character of $T$. Denote by $V_\chi$ and $\mathfrak{g}_\chi$ the $\chi$-weight spaces with respect to the actions given by $v \mapsto \rho(t)^{-1}t.v$ and $\mathrm{Ad}_{\rho(t)^{-1}}$, respectively. Then, by the above discussion, we see that
	$$
		T_{\pi(y)}(Y/G)_\chi \cong \coker(\mathfrak{g}_\chi \into V_\chi). 
	$$

	\section{Computation of the Weight Spaces} \label{s:wtsp}
	
	We now return to the case of the action of a torus on a moduli space of stable quiver representations. Let $Q$ be a quiver, $d$ a dimension vector and $\theta$ a stability condition. In the notation of the previous section, we take $V = R(Q,d)$, $Y = R(Q,d)^{\theta-\st}$, $X = M^{\theta-\st}(Q,d)$, $G = PG_d$, and $T = \G_m^n$ which acts via the characters $w_a$ as described in Section \ref{s:torusActions}.
	
	Let us analyze the action on the tangent space to a $T$-fixed point $[M]$. We have $\g = \gl_d^0 = \{ (x_i)_i \in \bigoplus_i \gl_{d_i} \mid \sum_i \tr(x_i) = 0 \}$ and
	\begin{align*}
		d_e \act_G(\blank,M): \g \to R(Q,d),\ x \mapsto [x,M] = (x_jM_a - M_ax_i)_{a:i \to j}.
	\end{align*}
	Suggestively, we are going to denote $[\blank,M] := d_e \act_G(\blank,M)$ and $[U,M] := \{ N \mid N = [x,M] \text{ for some } x \in U \}$ for a linear subspace $U \sub \mathfrak{g}$.
	Hence $T_{[M]}M^{\theta-\st}(Q,d) = R(Q,d)/[\g,M]$. 
	Decompose
	\begin{align*}
		R(Q,d) &= \bigoplus_{a \in Q_1} \bigoplus_{\xi \in \Z^n} \bigoplus_{\eta \in \Z^n} \Hom(V_{i,\xi},V_{j,\eta}) \\
		\mathfrak{g} &= \Big(\bigoplus_{i \in Q_0} \bigoplus_{\xi \in \Z^n} \bigoplus_{\xi' \in \Z^n} \Hom(V_{i,\xi},V_{i,\xi'})\Big)^0 \\
		&= \Big\{ (x_{i,\xi',\xi})_{i,\xi',\xi} \in \bigoplus_{i \in Q_0} \bigoplus_{\xi \in \Z^n} \bigoplus_{\xi' \in \Z^n} \Hom(V_{i,\xi},V_{i,\xi'}) \mid \sum_i \sum_\xi \tr(x_{i,\xi,\xi}) = 0 \Big\}
	\end{align*}
	In the second equation, the zero in the exponent stands for vanishing total trace; note that for $\xi \neq \xi'$, this is automatic as all entries are off-diagonal.
	Let $N_a = \sum_{\xi,\eta} N_{a,\eta,\xi}$ with $N_{a,\eta,\xi}:V_{i,\xi} \to V_{j,\eta}$. Then we have for $v = \sum_\xi v_{\xi} \in V_i$
	\begin{align*}
		\left( w_a(t)\rho_j(t)^{-1}N_a\rho_i(t) \right)(v) &= \sum_\xi w_a(t)\rho_j(t)^{-1}N_a(\xi(t)v_\xi) \\
		&= \sum_{\xi,\eta} w_a(t)\xi(t)\rho_j(t)^{-1}N_{a,\eta,\xi}(v_\xi) \\
		&= \sum_{\xi,\eta} w_a(t)\xi(t)\eta(t)^{-1}N_{a,\eta,\xi}(v_\xi) \\
		&= \sum_{\xi,\eta} (\xi - \eta + w_a)(t) N_{a,\eta,\xi}(v_\xi)
	\end{align*}
	and, similarly, for $x_i = \sum_{\xi,\xi'} x_{i,\xi',\xi} \in \mathfrak{gl}(V_i)$,
	\begin{align*}
		(\rho_i(t)^{-1}x_i\rho_i(t))(v) &= \sum_\xi \rho_i(t)^{-1}x_i(\xi(t)v_\xi) \\
		&= \sum_{\xi,\xi'} (\xi-\xi')(t)x_{i,\xi',\xi}(v_\xi).
	\end{align*}
	We have proved that $\smash{(T_{[M]}M^{\theta-\st}(Q,d))_\chi}$ coincides with 
	$$
		\coker \bigg( \Big( \bigoplus_{i \in Q_0} \bigoplus_{\xi \in \Z^n} \Hom(V_{i,\xi},V_{i,\xi-\chi}) \Big)^0 \into \Big( \bigoplus_{a \in Q_1} \bigoplus_{\xi \in \Z^n} \Hom(V_{s(a),\xi},V_{t(a),\xi+w_a-\chi}) \Big) \bigg).
	$$
	From this we can read off at once the dimension of the $\chi$-weight spaces. It is given in terms of the Euler form of $Q(w)$ by the following formula:
	\begin{align*}
		\dim  \big( T_{[M]}M^{\theta-\st}(Q,d) \big)_\chi &= \sum_{a \in Q_1} \sum_\xi \beta_{s(a),\xi}\beta_{t(a),\xi+w_a-\chi} - \sum_{i \in Q_0} \sum_\xi \beta_{i,\xi}\beta_{i,\xi-\chi} + \delta(\chi,0) \\
		&= \delta(\chi,0) - \langle \beta,s_{-\chi}(\beta) \rangle_{Q(w)}.
	\end{align*}
	
	Now recall that for any two representations $M$ and $N$ of a quiver $Q$ there exists an exact sequences
	$$
		0 \to \Hom_Q(M,N) \to \bigoplus_{i \in Q_0} \Hom(M_i,N_i) \xto{}{\phi_{M,N}} \bigoplus_{a \in Q_1} \Hom(M_{s(a)},N_{t(a)}) \to \Ext^1_Q(M,N) \to 0
	$$
	which arises from the standard projective resolution of $M$ (see \cite{Crawley-Boevey:92}).
	The map $\phi_{M,N}$ sends a tuple $(A_i)_i$ to $(A_{t(a)}M_a - N_aA_{s(a)})_a$. It is called the Happel--Ringel map. Let us consider the Happel--Ringel map for the quiver $Q(w)$ and a lift $N$ of $M$ to $Q(w)$ which corresponds to the choice of the weight space decompositions $V_i = \bigoplus_\chi V_{i,\chi}$. The Happel--Ringel map
	$$
		\phi_{N,s_{-\chi}(N)}: \bigoplus_{i \in Q_0} \bigoplus_{\xi \in \Z^n} \Hom(V_{i,\xi},V_{i,\xi-\chi}) \to \bigoplus_{a \in Q_1} \bigoplus_{\xi \in \Z^n} \Hom(V_{s(a),\xi},V_{t(a),\xi+w_a-\chi})
	$$
	restricted to tuples with vanishing total trace equals the restriction of the map $[\blank,M]$ to the weight spaces of weight $\chi$. This shows that $\smash{(T_{[M]}M^{\theta-\st}(Q,d))_\chi}$ is isomorphic to the extension group $\Ext_{Q(w)}^1(N,s_{-\chi}(N))$. We summarize our findings in the following
	
	\begin{thm} \label{t:wtsp}
		Let $[M]$ be a $T$-fixed point of $M^{\theta-\st}(Q,d)$ and let $N$ be a lift of $M$ to $Q(w)$.
		The $\chi$-weight space of the tangent space of $M^{\theta-\st}(Q,d)$ to $[M]$ is isomorphic to
		$$
			(T_{[M]}M^{\theta-\st}(Q,d))_\chi \cong \Ext_{Q(w)}^1(N,s_{-\chi}(N)).
		$$
		With $\beta = \dimvect N$ we obtain
		\begin{align*}
			\dim  \big( T_{[M]}M^{\theta-\st}(Q,d) \big)_\chi &= \delta(\chi,0) - \langle \beta,s_{-\chi}(\beta) \rangle_{Q(w)}.
		\end{align*}
	\end{thm}

	\section{Actions of Rank 1 Tori on Smooth Projective Quiver Moduli}\label{s:attr}
	
	Let $Q$ be an acyclic quiver, $d$ a dimension vector, and $\theta$ be a stability condition such that $d$ is $\theta$-coprime. Let $\G_m$ act on $M^\theta(Q,d)$ by weights $w_a \in \Z$. Let $Q(w)$ be the corresponding covering quiver. Later, in the relevant examples we will take a $\G_m$-action which is given by weights $w_{a_1} \gg \ldots \gg w_{a_N} > 0$ for a given order $\{a_1,\ldots,a_N\} = Q_1$ on the arrows.

	First we state an immediate consequence of the description of the dimensions of the weight spaces in Theorem \ref{t:wtsp}.
	
	\begin{cor} \label{c:poincare}
		Let $k = \C$.
		%Let $Q$ be an acyclic quiver, $\theta$ be a stability condition, and $d$ be a $\theta$-coprime dimension vector of $Q$. Let $T = \G_m$ act by the weights $w_a \in \Z$. 
		The Poincar\'{e} polynomial of $M^\theta(Q,d)$ equals
		$$
			P_{M^\theta(Q,d)}(t) = \sum_\beta t^{2d_\beta} P_{M^\theta(Q(w),\beta)}(t),
		$$
		the sum ranging over all equivalence classes of dimension vectors $\beta$ of $Q(w)$ which are compatible with $d$. The exponent $d_\beta$ is determined by the formula
		$$
			d_\beta = - \sum_{n > 0} \langle \beta,s_{-n}(\beta) \rangle_{Q(w)}.
		$$
	\end{cor}
	
	Our next goal is to give an explicit description of the attractors as subsets of $M^\theta(Q,d)$. For this we need to introduce a piece of notation. Let $N \in R(Q,d)$. Assume that for every $i \in Q_0$, there exists a filtration 
	$$
		\ldots \sub F_{i,n} \sub F_{i,n+1} \sub \ldots \sub V_i
	$$ 
	in such a way that $N_a(F_{i,n}) \sub F_{j,n+w_a}$ for every arrow $a: i \to j$. Such a collection $F_*$ of filtrations will be called an $w$-twisted filtration of $N$. Note that $F_*$ is not a filtration of $N$ by subrepresentations. But the induced maps $\bar{N}_a: F_{i,n}/F_{i,n-1} \to F_{j,n+w_a}/F_{j,n+w_a-1}$ give a representation of $Q(w)$ which we denote $\smash{\gr^{F_*}}(N)$.

	\begin{prop}
		Let $[M] \in M^\theta(Q,d)^{\G_m}$ and $[N] \in  M^\theta(Q,d)$. Then $[N] \in\Att^+([M])$ if and only if there is exists an $w$-twisted filtration $F_*$ of $N$ such that $\smash{\gr^{F_*}}(N) \cong M$ as representations of $Q$.
	\end{prop}
	
	%\todo[inline]{I changed the moniker. I think it should rather be a proposition because it's just an extension of Ryan and Thorsten's result that they somehow didn't see.}
	
	\begin{proof}
		Assume that there exists an $w$-twisted filtration $F_*$ such that $\gr^{F_*}N \cong M$. We have to show $\lim_{t\rightarrow 0} t.[N] = [M]$. Abbreviate $V_{i,n}:=F_{i,n}/F_{i,n-1}$. Let $\psi: \G_m \to G_d$ be the one-parameter subgroup which corresponds to the decomposition $V_i = \bigoplus_{n\in\Z} V_{i,n}$; concretely this means for $v_i = \sum_n v_{i,n}$ with $v_{i,n} \in V_{i,n}$ that $\psi_i(t)v_i = \sum_n t^nv_{i,n}$.
		For an arrow $a:i \to j$ and $m, n \in \Z$ we define $N_{a,m,n}: V_{i,n}\rightarrow V_{j,m}$ as the linear map such that $N_a(v_i) = \smash{\sum_{m,n}} N_{a,m,n}(v_{i,n})$ for $v_i = \smash{\sum_n} v_{i,n}$ with $v_{i,n} \in V_{i,n}$. Then $N_{a,m,n}=0$ whenever $m>n+w_a$ as $N_a(F_{i,n}) \sub F_{j,n+w_a}$. Now consider $\psi(t)^{-1}\cdot t.N$; it lies in the same $G_d$-orbit as $t.N$. Since for $v_{i,n} \in V_{i,n}$
		\begin{align*}
			(\psi(t)^{-1}\cdot t.N)_a(v_{i,n}) &= \psi_j(t)^{-1} t^{w_a}N_{\alpha}(\psi_i(t) v_{i,n})\\
			&= \psi_j(t)^{-1}\sum_m(t^{n+w_a}N_{\alpha,m,n}(v_{i,n}))\\
			&= \sum_m t^{n-m+w_a}N_{\alpha,m,n}(v_{i,n})
		\end{align*}
		These considerations show that $\lim_{t \to 0} t.[N]$ is represented by $\gr^{F_*}(N)$, considered as a representation of $Q$. But $\smash{\gr^{F_*}}(N) \cong M$.

		Assume conversely that $\lim_{t\rightarrow 0} t.[N] = [M]$. Let $\dot{\rho}:\G_m \to G_d$ be a lift of the homomorphism $\rho: \G_m \to PG_d$ which corresponds to $[M]$. An application of \cite[Prop.\ 5.11]{KW:19}, which is a consequence of Hilbert's criterion, shows that $\lim_{t \to 0} \dot{\rho}(t)^{-1}\cdot t.gN = M$ for some $g \in G_d$. Let $V_i = \bigoplus_{n \in \Z} V_{i,n}$ be the weight space decomposition induced by $\dot{\rho}$ and abbreviate $\smash{\tilde{N}} := gN$. Let $\smash{\tilde{N}_{a,m,n}}: V_{i,n} \to V_{j,m}$ be defined in the same vein as before. For $v_{i,n} \in V_{i,n}$ we obtain
		$$
			(\dot{\rho}(t)^{-1}\cdot t.\tilde{N})(v_{i,n}) = \sum_m t^{n-m+w_a}\tilde{N}_{a,m,n}(v_{i,n}).
		$$
		We define $M_{a,n}: V_{i,n} \to V_{j,n+w_a}$ as the restriction of $M_a$.
		As we know that this converges to $M$ as $t \to 0$, we deduce that $\smash{\tilde{N}_{a,m,n}} = 0$ whenever $m > n-w_a$ and that $\smash{\tilde{N}_{a,n+w_a,n}} = \smash{M_{a,n}}$. Setting $\smash{\tilde{F}_{i,n}} = \bigoplus_{m \leq n} V_{i,m}$ we see that $\smash{\tilde{F}_*}$ is an $w$-twisted filtration for $\smash{\tilde{N}} = gN$ such that $\smash{\gr^{\tilde{F}_*}(\tilde{N})} \cong M$. Hence $F_* := \smash{g^{-1}\tilde{F}_*}$ is the desired $w$-twisted filtration of $N$.
	\end{proof}
	
	\begin{rem}
		We would like to point out that the above proposition is essentially a reformulation of a Theorem \cite[Thm.\ 5.15]{KW:19} due to Kinser and Weist. They construct a space $V_T$ in the proof of the mentioned theorem which, if translated to the language of $w$-twisted flags, is just what we obtain.
	\end{rem}
	
	Let $[M] \in M^\theta(Q,d)$ be a fixed point and let $V_i = \bigoplus_n V_{i,n}$ be the the corresponding decompositions for $i \in Q_0$ and define the flag $F_*$ by $F_{i,n} = \bigoplus_{m \leq n} V_{i,m}$. For $N \in R(Q,d)$ define $N_{a,m,n}: V_{i,n} \to V_{j,m}$ as in the proof of the previous proposition. For $g \in G_d$ define $g_{i,m,n}: V_{i,n} \to V_{i,m}$ by $g_iv_{i,n} = \sum_m g_{i,m,n}v_{i,n}$ for $v_{i,n} \in V_{i,n}$. Set
	\begin{align*}
		R_{F_*}(M) &= \{ N \in R(Q,d) \mid N_{a,m,n} = 0 \text{ for } m > n+w_a \text{ and } N_{a,n+w_a,n} = M_{a,n} \text{ for all $a \in Q_1$} \} \\
		U_{F_*} &= \{ g \in G_d \mid g_{i,m,n}=0 \text{ for } m > n \text{ and } g_{i,n,n} = \id_{V_i} \text{ for all $i \in Q_0$} \}
	\end{align*}
	Note that $N \in R_{F_*}(M)$ is automatically stable.
	The group $\smash{U_{F_*}}$ is unipotent and connected, it is up to a factor of $\G_m$ the same group as in \cite[(5.26)]{KW:19}. It acts on $\smash{R_{F_*}(M)}$ with trivial stabilizers. The associated fiber bundle $G_d \times^{U_{F_*}} R_{F_*}(M)$ agrees with $\pi^{-1}(\Att^+([M]))$ by the above proposition. Thus $\Att^+([M]) \cong R_{F_*}(M)/U_{F_*}$ --- a principal $U_{F_*}$-fiber bundle. Theorem \ref{t:bb} tells us that $\Att^+([M])$ is a vector space, isomorphic to $(T_{[M]}M^\theta(Q,d))_+$. In particular, it is isomorphic to 
	$$
		T_{[M]}\Att^+([M]) \cong T_{[M]}(R_{F_*}(M)/U_{F_*}) \cong T_{M}(R_{F_*}(M))/\im d_e \act_{U_{F_*}}(\blank,M)
	$$
	Define 
	$$
		R_{F_*} = R_{F_*}(0) = \bigoplus_{a: i \to j} \bigoplus_{m < n+w_a} \Hom(V_{i,n},V_{j,m}). 
	$$
	Then $R_{F_*}(M) = \{M\} + R_{F_*}$, so $T_{M}(R_{F_*}(M)) = R_{F_*}$. The derivative $d_e \act_{U_{F_*}}(\blank,M):\mathfrak{u}_{F_*} \to R_{F_*}$ of the action map sends $x = (x_a)_{a \in Q_1}$ to $[x,M]$ (c.f.\ Section \ref{s:wtsp}). It is an injective map as the action is free. Hence
	$$
		T_{[M]}\Att^+([M]) \cong R_{F_*}/[\mathfrak{u}_{F_*},M].
	$$	
	The Lie algebra of $U_{F_*}$ is, explicitly, 
	$$
		\mathfrak{u}_{F_*} = \bigoplus_{i \in Q_0} \bigoplus_{m < n} \Hom(V_{i,n},V_{i,m}).
	$$
	The derivative of the action map is compatible with the weight space decompositions $\mathfrak{u}_{F_*} = \bigoplus_{k > 0} \mathfrak{u}_{F_*,k}$ and $R_{F_*} = \bigoplus_{k > 0} R_{F_*,k}$ where
	\begin{align*}
		\mathfrak{u}_{F_*,k} &= \bigoplus_{i \in Q_0} \bigoplus_{n \in \Z} \Hom(V_{i,n},V_{i,n-k}) &
		R_{F_*,k} &= \bigoplus_{a: i \to j} \bigoplus_{n \in \Z} \Hom(V_{i,n},V_{j,n+w_a-k}).
	\end{align*}
	
	\begin{comment}
	\todo[inline]{Is the following Lemma true?!}
	
	\begin{lem}
		Let $U$ be a unipotent linear algebraic group, let $V$ be a finite-dimensional vector space. Assume that $U$ acts freely algebraically on $V$, which we regard as an affine space. Assume that $V/U$ is also an affine space. Fix $x \in V$ and a vector space complement $V' \sub V$ of $\im d_e \act_U(\blank,x)$. Is then the composition $\{x\}+V' \into V \to V/U$ an isomorphism?
	\end{lem}
	
	\todo[inline]{I think we can also prove the following even without the lemma}
	\end{comment}
	
	\begin{thm} \label{t:att_explicit}
		Let $[M] \in M^\theta(Q,d)^{\G_m}$, let $V_i = \bigoplus_{n \in \Z} V_{i,n}$ be a corresponding weight space decomposition and let $F_{i,n} = \smash{\bigoplus_{m \leq n}} V_{i,m}$.
		For every $k > 0$ let $R'_k$ be a vector space complement of $[\mathfrak{u}_{F_*,k},M]$ inside $R_{F_*,k}$. Then the composition $\{M\} + \bigoplus_{k > 0} R'_k \into R_{F_*}(M) \to \Att^+(M)$ is an isomorphism.
	\end{thm}
	
	\begin{proof}
		Abbreviate $R' = \bigoplus_{k>0} R'_k$. We show
		\begin{enumerate}
			\item No two elements of $\{M\} + R'$ are in the same $U_{F_*}$-orbit.
			\item There is a morphism $\gamma: R_{F_*} \to U_{F_*}$ such that $\gamma(N)N \in \{M\} + R'$ for every $N \in R_{F_*}(M)$.
		\end{enumerate}
		This does indeed show that $\{M\} + R' \to \Att^+([M])$ is an isomorphism. The inverse is obtained as follows. As $R_{F_*}(M) \to \Att^+([M])$ is a principal $U_{F_*}$-fiber bundle and as the group $U_{F_*}$ is special, we find local sections $s_i: U_i \to R_{F_*}(M)$ on a Zariski open cover $\{U_i\}$. The morphisms $U_i \to \{M\} + R'$ given by $\gamma(s_i([N]))s_i([N])$ glue. The obtained morphism is the desired inverse.
		To show the two assertions above, we make a general observation.
		Let $N = M+x, N' = M+\smash{x'} \in \smash{R_{F_*}(M)}$ and $g \in \smash{U_{F_*}}$. Define again $\smash{N_{a,m,n}}$, $\smash{N'_{a,m,n}}$, $\smash{g_{i,m,n}}$ and $\smash{x_{i,m,n}}$, $\smash{x'_{a,m,n}}$ as before. Let $h_i = \smash{g_i^{-1}}$ and decompose it into $h_{i,m,n}$. If $gN = N'$ then
		\begin{align*}
			N'_{a,m,n} &= 
			\sum_{r = m}^{n+w_a} \sum_{s = r-w_a}^n g_{j,m,r}N_{a,r,s}h_{i,s,n}
		\end{align*}
		for all $a: i \to j$ and all $m,n \in \Z$. The ranges of the sums come from $N$, $g$, and $x$ being adapted to the filtration $F_*$. For $m > n+w_a$ this equation is void and the equation for $m = n+w_a$ is tautological. The other equations are
		\begin{align*}
			x'_{a,n+w_a-1,n} =\ & M_{a,n-1}h_{i,n-1,n} + x_{a,n+w_a-1,n} + g_{j,n+w_a-1,n+w_a}M_{a,n} \\
			x'_{a,n+w_a-2,n} =\ & M_{a,n-2}h_{i,n-2,n} + x_{a,n+w_a-2,n-1}h_{i,n-1,n} + x_{a,n+w_a-2,n} + \\
					&+ g_{j,n+w_a-2,n+w_a-1}M_{a,n-1}h_{i,n-1,n} + g_{j,n+w_a-2,n+w_a-1}x_{a,n+w_a-1,n} + \\ 
					&+ g_{j,n+w_a-2,n+w_a}M_{a,n} \\
					%&= M_{a,n-2}h_{i,n-2,n} + \tilde{x}_{n+w_a-2,n} + g_{j,n+w_a-2,n+w_a}M_{a,n} \\
					\vdots\ \ & \\
			x'_{a,n+w_a-k,n} =\ & M_{a,n-k}h_{i,n-k,n} + \ldots + g_{j,n+w_a-k,n+w_a}M_{a,n}
		\end{align*}
		Note that for $m < n$ the $h_{i,m,n}$ satisfy the equations
		$$
			h_{i,m,n} + \sum_{r=m+1}^{n-1} g_{i,m,r}h_{i,r,n} + g_{i,m,n},
		$$
		so we may write $h_{i,m,n} = -g_{i,m,n} + \tilde{h}_{i,m,n}$, where $\tilde{h}_{i,m,n}$ depends algebraically on elements $g_{i,r,s}$ for which $s-r < n-m$. This shows that we obtain an equation
		$$
			x'_{a,n+w_a-k,n} = -M_{a,n-k}g_{i,n-k,n} + \tilde{x}_{a,n+w_a-k,n} + g_{j,n+w_a-k,n+w_a}M_{a,n}
		$$
		where $\tilde{x}_{a,n+w_a-k,n}$ is an expression that depends algebraically only on $x_{a,n+w_a-k,n}$, $M$, and $x_{a,r+w_a,s}$, $g_{i,r,s}$, $g_{j,r,s}$ with $s-r < k$. Define
		\begin{align*}
			x'_k &= \sum_a \sum_n x'_{a,n+w_a-k,n} \in R_{F_*,k}, \\
			\tilde{x}_k &= \sum_a \sum_n \tilde{x}_{a,n+w_a-k,n} \in R_{F_*,k}, \\
			g_k &= \sum_i \sum_n g_{i,n,n-k} \in \mathfrak{u}_{F_*,k}.
		\end{align*}
		Then $\tilde{x}_k$ depends algebraically only on $x_k$, $M$, and $x_l$, $g_l$ with $l < k$. Moreover, the equation
		\begin{equation} \label{e:star} \tag{*}
			x'_k = [M,g_k] + \tilde{x}_k
		\end{equation}
		holds.
		With these observations we are able to show (1) and (2).
		\begin{enumerate}
			\item Suppose that both $N$ and $N'$ lie in $\{M\}+R'$. We show by induction on $k>0$ that $\tilde{x}_k = x_k$ and $g_k=0$. For $k=1$ the first equality is automatic. Equation \ref{e:star} implies $x'_1 - x_1 = [M,g_1]$ which is zero, as they live in complementary subspaces. By injectivity of the map $[M,\blank]$, we obtain $g_1 = 0$. We assume $k>0$. From the definition of $\tilde{x}_k$ we see that $g_l = 0$ (all $l<k$) forces $\tilde{x}_k = x_k$. We conclude $g_k = 0$ just as in the base of the induction.
			\item Now suppose $N = M+x \in R_{F_*}(M)$ is given. We would like to find $g = \gamma(N) \in U_{F_*}$ depending algebraically on $N$ such that $gN \in \{M\}+R'$. We will construct by induction on $k > 0$ matrices $g_k \in \mathfrak{u}_{F_*,k}$ such that $g = e + \sum_{k > 0} g_k$ does the job. For $k = 1$, we want $g_1$ such that $[M,g_1] + x_1$ lies in $R'_1$. By assumption, such a $g_1$ exists and is unique. The dependency of $g_1$ on $x_1$ is algebraic, as $g_1$ is the image of $x_1$ under the projection $R_{F_*} \to \mathfrak{u}_{F_*}$ along $R'_1$. The induction step $k>0$, we need $g_k$ such that $[M,g_k]+\tilde{x}_k \in R'_k$. The element $\tilde{x}_k$ depends algebraically on $x$ and the matrices $g_l$ constructed before. Hence we may argue as in the case $k=1$. \qedhere
		\end{enumerate}
	\end{proof}
	
	\begin{rem}
		Theorem \ref{t:att_explicit} shows that the principal $U_{F_*}$-fiber bundle $\smash{R_{F_*}(M)} \to \Att^+([M])$ admits a global section. It is therefore trivial. In particular, $\smash{R_{F_*}(M)} \to \Att^+([M])$ is a (trivial) affine bundle. This answers a question posed in \cite{KW:19}.
	\end{rem}
	
	%\todo[inline]{Remark about question in KW and conceptual meaning of theorem (trivial torsor for unipotent group)}
	
	\begin{defn*}
		Let $G$ be an algebraic group acting on a variety $Y$ and suppose that there exists a geometric quotient $\pi: Y \to X = Y/G$. A generic normal form for $X$ or for the action of $G$ on $Y$ is a locally closed embedding $\phi: \A^n \to Y$ such that the composition $\pi\phi: \A^n \to X$ is an open embedding with a dense image.
	\end{defn*}

	\begin{cor} \label{c:gen_nf}
		%Let $Q$ be an acyclic quiver, $\theta$ be a stability condition, and $d$ be a $\theta$-coprime dimension vector of $Q$.
		Suppose there exists a dimension vector $\beta \in \Lambda_+(Q(w))$ with $c(\beta) = d$ satisfying the following two conditions:
		\begin{enumerate}
			\item $\beta$ is a real root of $Q(w)$, i.e.\ $\langle \beta,\beta \rangle_{Q(w)} = 1$,
			\item $\langle \beta,s_{-n}(\beta) \rangle_{Q(w)} = 0$ for all $n < 0$.
		\end{enumerate}
		Let $[M]$ be the isolated fixed point which corresponds to $\beta$. Let $R'_k$ be chosen as in Theorem \ref{t:att_explicit}. Then the morphism
		$$
			\phi: \bigoplus_{k>0} R'_k \to R(Q,d)^\theta,\ x \mapsto x+M
		$$
		is a generic normal form for $M^\theta(Q,d)$.
	\end{cor}
	
	\begin{proof}
		As $\beta$ is a real root, the moduli space $M^{\hat{\theta}}(Q(w),\beta)$ is a single point, so $[M]$ is indeed an isolated fixed point. The second condition implies 
		\begin{align*}
			\dim M^\theta(Q,d) = 1 - \langle d,d \rangle_Q = 1-\sum_{n \in \Z} \langle \beta,s_{-n}(\beta) \rangle_{Q(w)} = \sum_{n>0} \langle \beta,s_{-n}(\beta) \rangle_{Q(w)} = \dim \Att^+([M])
		\end{align*}
		where we have used Theorem \ref{t:wtsp} to obtain the rightmost equality. As $M^\theta(Q,d)$ is irreducible and the attractors are locally closed, this tells us that $\Att^+([M])$ is the unique open attractor. An application of Theorem \ref{t:att_explicit} concludes the proof.
	\end{proof}

	\begin{rem}
		It was shown by Schofield that the moduli space $M^\theta(Q,d)$ is a rational variety \cite[Thm.\ 6.4]{Schofield:01}. The above corollary gives a sufficient condition for when there exists an open subset of $M^\theta(Q,d)$ which is isomorphic to an affine space (and gives an explicit description of that subset). It would be interesting to know whether there is always an open subset which is isomorphic to an affine space.
	\end{rem}

	\section{Application to Kronecker Moduli}\label{s:appl}
	
	%\todo[inline]{Check again, all of it!}
	
	In this section let $Q = K(l+1)$ be the $l+1$-Kronecker quiver. It has two vertices $i$ and $j$ and $l+1$ arrows $a_1,\ldots,a_{l+1}$. We are concerned with dimension vectors of the form $d = (2,2r+1)$. Take the stability condition $\theta$ defined by $\theta(e) = e_i$. A representation $M \in R(Q,d)$ is a tuple $(A_1,\ldots,A_{l+1})$ of $(2r+1) \times 2$-matrices; it is (semi\nobreakdash-)stable if and only if $\im A_1 + \ldots + \im A_{l+1} = k^{2r+1}$ and $\dim \langle A_1x,\ldots,A_{l+1}x \rangle \geq r+1$ for every $x \in k^2 - \{0\}$. Obviously, for $M^\theta(K(l+1),d) \neq \emptyset$, we have to require $l \geq r$. We consider the action of $T = \G_m$ by the weights $w_1 \gg \ldots \gg w_{l+1} > 0$ (we write $w_i$ instead of $w_{a_i}$ for brevity).
	There are two types of covering dimension vectors:
	$$
		
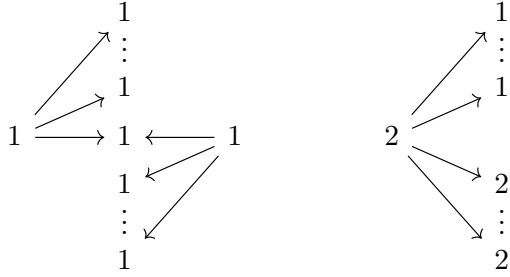
\begin{tikzcd}
				& 1	& \\[-2.5em]
				& \vdots& \\[-2em]
				& 1	& \\[-1.5em]
			1 \arrow{r}{} \arrow{ur}{} \arrow{uuur}{}	& 1	& 1 \arrow{l}{} \arrow{dl} \arrow{dddl} \\[-1.5em]
				& 1	& \\[-2.5em]
				& \vdots& \\[-2em]
				& 1	&
		\end{tikzcd}
		\quad \quad \qquad
		\begin{tikzcd}
				& 1 \\[-2.5em]
				& \vdots& \\[-2em]
				& 1	& \\[-1.5em]
			2 \arrow{ur} \arrow{uuur} \arrow{dr} \arrow{dddr}	&	& \\[-1.5em]
				& 2	& \\[-2.5em]
				& \vdots& \\[-2em]
				& 2	&
		\end{tikzcd}
	$$
	The underlying vertices of the first type of covering dimension vector are
	$$
		
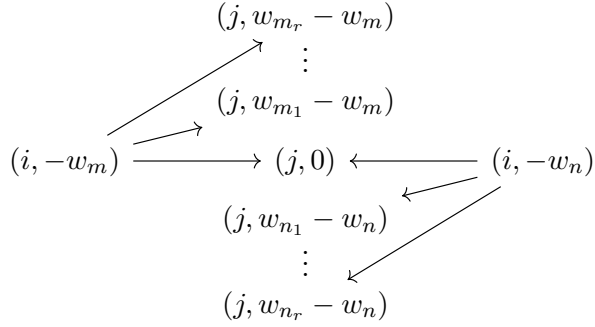
\begin{tikzcd}
				& (j,w_{m_r}-w_{m})	& \\[-2.5em]
				& \vdots& \\[-2em]
				& (j,w_{m_1}-w_{m})	& \\[-1.5em]
			(i,-w_{m}) \arrow{r}{} \arrow{ur}{} \arrow{uuur}{}	& (j,0)	& (i,-w_{n}) \arrow{l}{} \arrow{dl} \arrow{dddl} \\[-1.5em]
				& (j,w_{n_1}-w_{n})	& \\[-2.5em]
				& \vdots& \\[-2em]
				& (j,w_{n_r}-w_{n})	&
		\end{tikzcd}
	$$
	for two sequences of pairwise distinct indexes $m_1,\ldots,m_r$ and $n_1,\ldots,n_r$ in $\{1,\ldots,l+1\}$ and indexes $m \neq n$ such that $m \notin \{m_1,\ldots,m_r\}$ and $n \notin \{n_1,\ldots,n_r\}$. By symmetry we may require $m_1 < \ldots < m_r$ and $n_1 < \ldots < n_r$ and $m < n$. Let $1 \leq m_{r+1} < \ldots < m_l \leq l+1$ and $1 \leq n_{r+1} < \ldots < n_l \leq l+1$ be the complementary indexes to $m,m_1,\ldots,m_r$ and $n,n_1,\ldots,n_r$. Denote this dimension vector by $\beta$ and let $[M_\beta]$ be the corresponding fixed point; note that $[M_\beta]$ is an isolated fixed point because $\beta$ is a real root. Abbreviate $V = T_{[M_\beta]}M^\theta(K(l+1),d)$. Then for any $\chi \in \Z$ the weight space dimension is given by Theorem \ref{t:wtsp} as
	\begin{align*}
		\dim V_\chi =&~ \delta(\chi,0) - \langle \beta,s_{-\chi}(\beta) \rangle \\
			=&~ \delta(\chi,0) - \beta_{(i,-w_{m}-\chi)} - \beta_{(i,-w_{n}-\chi)} - \beta_{(j,-\chi)} - \sum_{\nu=1}^r ( \beta_{(j,w_{m_\nu}-w_{m}-\chi)} + \beta_{(j,w_{n_\nu}-w_{n}-\chi)} )\\ 
			&+ \sum_{t=1}^{l+1} ( \beta_{(j,w_t-w_{m}-\chi)} + \beta_{(j,w_t-w_{n}-\chi)} ) \\
			=&~ \delta(\chi,0) - \beta_{(i,-w_{m}-\chi)} - \beta_{(i,-w_{n}-\chi)} + \beta_{(j,-\chi)} + \sum_{\mu=r+1}^l ( \beta_{(j,w_{m_\mu}-w_{m}-\chi)} + \beta_{(j,w_{n_\mu}-w_{n}-\chi)} ).
	\end{align*}
	The picture above gives us the entries of the dimension vector $\beta$. We compute
	\begin{align*}
		\dim V_\chi =& -\delta(\chi,w_{n}-w_{m}) - \delta(\chi,w_{m}-w_{n}) \\
			&+ \sum_{\nu=1}^r (\delta(\chi,w_m-w_{m_\nu}) + \delta(\chi,w_n-w_{n_\nu})) \\
			&+ \sum_{\mu=r+1}^l ( \delta(\chi,w_{m_\mu}-w_{m}) + \delta(\chi,w_{n_\mu}-w_{n}) ) \\
			&+ \sum_{\mu=r+1}^l \sum_{\nu=1}^r ( \delta(\chi,w_{m_\mu}-w_{m_\nu})+\delta(\chi,w_{m_\mu}-w_{m}-w_{n_\nu} +w_{n}) ) \\ 
			&+ \sum_{\mu=r+1}^l \sum_{\nu=1}^r ( \delta(\chi,w_{n_\mu}-w_{n_\nu}) + \delta(\chi,w_{n_\mu}-w_{n}-w_{m_\nu}+w_{m}) ).
	\end{align*}
	%Let $\lambda$ be a one-parameter subgroup with $\lambda_1 \gg \ldots \gg \lambda_{l+1}$. 
	Then we obtain the following value for the dimension of the attractor of $[M_\beta]$:
	\begin{align*}
		\dim \Att^+([M_\beta]) = -1 
			&+ \sharp \{ \nu \in \{1,\ldots,r\} \mid m < m_\nu \} \\
			&+ \sharp \{ \nu \in \{1,\ldots,r\} \mid n < n_\nu \} \\
			&+ \sharp \{ \mu \in \{r+1,\ldots,l\} \mid m_\mu < m \} \\
			&+ \sharp \{ \mu \in \{r+1,\ldots,l\} \mid n_\mu < n \} \\
			&+ \sharp \{ (\mu,\nu) \in \{ r+1,\ldots,l\} \times \{1,\ldots,r\} \mid m_\mu < m_\nu \} \\
			&+ \sharp \{ (\mu,\nu) \in \{ r+1,\ldots,l\} \times \{1,\ldots,r\} \mid \min(m_\mu,n) < \min(n_\nu,m) \} \\
			&+ \sharp \{ (\mu,\nu) \in \{ r+1,\ldots,l\} \times \{1,\ldots,r\} \mid n_\mu < n_\nu \} \\
			&+ \sharp \{ (\mu,\nu) \in \{ r+1,\ldots,l\} \times \{1,\ldots,r\} \mid \min(n_\mu,m) < \min(m_\nu,n) \}.
			%= -1 &+ \sum_{\nu=1}^r (m_\nu+n_\nu-2\nu) \\
			%&+ \sharp \{ (\mu,\nu) \in \{ r+1,\ldots,l\} \times \{1,\ldots,r\} \mid m_\mu < \min(n_\nu,m) \} \\
			%&+ \sharp \{ (\mu,\nu) \in \{ r+1,\ldots,l\} \times \{1,\ldots,r\} \mid \min(n_\mu,m) < \min(m_\nu,n) \}
	\end{align*}
	Note that the sum of the first four cardinalities is greater than or equal to one.
	The attractor dimension of $[M_\beta]$ being uniquely determined by $m,m_* = (m_1,\ldots,m_r)$, $n$, and $n_* = (n_1,\ldots,n_r)$, we define $d_1^+(m,m_*,n,n_*) = \dim \Att^+([M_\beta])$. The dimension $d_1^-(m,m_*,n,n_*)$ of the attractor $\Att^-([M_\beta])$ is given by replacing every ``${<}$'' in the above formula by a ``${>}$''.
	
	\enlargethispage{2cm}
	\begin{prop}
		The moduli space $M^\theta(K(l+1),(2,2r+1))$ possesses a generic normal form.
		More precisely, setting $s=l-r$, there exists a generic normal form $\A^{(2s+1)(2r+1)-3} \to R(K(l+1),(2,2r+1))^\theta$ whose image consists precisely of those tuples of matrices which are of the form
		\begin{comment}
		$$
			(A_1,\ldots,A_s,A_{s+1},A_{s+2},A_{s+3},\ldots,A_{l+1})
		$$ 
		$$
			\begin{blockarray}{c(c(cc)c(cc)c(cc)c(cc)c(cc)c(cc)c)}
				1	& & *      & *      &		& *      & * 		&   &   &   &   &   & * 	&   &   &   &		& 1 &   & \\
				\vdots	& & \vdots & \vdots &		& \vdots & \vdots	&   &   &   &   &   & \vdots	&   &   &   &		&   &   & \\
				r-1	& & \vdots & \vdots &		& \vdots & \vdots	&   &   &   &   &   & * 	&   & 1 &   &		&   &   & \\
				r	& & \vdots & \vdots &		& \vdots & \vdots	&   &   &   &   & 1 &   	&   &   &   &		&   &   & \\
				r+1	& & \vdots & \vdots & ,\ldots,	& \vdots & \vdots	& , &   &   & , &   & *		& , &   &   &,\ldots,	&   & 1 & \\
				\vdots	& & \vdots & \vdots &		& \vdots & \vdots	&   &   &   &   &   & \vdots 	&   &   &   &		&   &   & \\
				2r-1	& & \vdots & \vdots &		& \vdots & \vdots	&   &   &   &   &   & * 	&   &   & 1 &		&   &   & \\
				2r	& & \vdots & \vdots &		& \vdots & \vdots	&   & 1 &   &   &   & 1 	&   &   &   &		&   &   & \\
				2r+1	& & *      & *      &		& *      & * 		&   &   & 1 &   &   &  		&   &   &   &		&   &   & \\
			\end{blockarray}\ .
		$$
		\end{comment}
		$$
			\begin{blockarray}{cccccccccccccc}
				& & A_1 & & A_s & & A_{s+1} & & A_{s+2} & & A_{s+3} & & A_{l+1} & \\
			\begin{block}{c(c(c)c(c)c(c)c(c)c(c)c(c)c)}
				1 & & \ray{*}{*} & & \ray{*}{*} & & \ray{}{} & & \ray{}{*} & & \ray{}{} & & \ray{1}{} & \\
				\vd & & \ray{\vd}{\vd} & & \ray{\vd}{\vd} & & \ray{}{} & & \ray{}{\vd} & & \ray{}{} & & \ray{}{} & \\
				r-1 & & \ray{\vd}{\vd} & & \ray{\vd}{\vd} & & \ray{}{} & & \ray{}{*} & & \ray{1}{} & & \ray{}{} & \\
				r & & \ray{\vd}{\vd} & & \ray{\vd}{\vd} & & \ray{}{} & & \ray{1}{} & & \ray{}{} & & \ray{}{} & \\
				r+1 & & \ray{\vd}{\vd} & ,\ldots, & \ray{\vd}{\vd} & , & \ray{}{} & , & \ray{}{*} & , & \ray{}{} & ,\ldots, & \ray{}{1} & \\
				\vd & & \ray{\vd}{\vd} & & \ray{\vd}{\vd} & & \ray{}{} & & \ray{}{\vd} & & \ray{}{} & & \ray{}{} & \\
				2r-1 & & \ray{\vd}{\vd} & & \ray{\vd}{\vd} & & \ray{}{} & & \ray{}{*} & & \ray{}{1} & & \ray{}{} & \\
				2r & & \ray{\vd}{\vd} & & \ray{\vd}{\vd} & & \ray{1}{} & & \ray{}{1} & & \ray{}{} & & \ray{}{} & \\
				2r+1 & & \ray{*}{*} & & \ray{*}{*} & & \ray{}{1} & & \ray{}{} & & \ray{}{} & & \ray{}{} & \\
			\end{block}
			\end{blockarray}\ .
		$$
	\end{prop}
	
	\begin{proof}
		By the above fomula for the attractor dimensions we see that $d_1^-(m,m_*,n,n_*) = 0$ if and only if $m=s+1$, $m_*=(s+2,\ldots,l+1)$, $n=s+2$, and $n_*=(s+1,s+3,\ldots,l+1)$. As the corresponding dimension vector $\beta$ is a real root, we know that a generic normal form exists whose image in $M^\theta(K(l+1),(2,2r+1))$ is $\Att^+([M_\beta])$. Corollary \ref{c:gen_nf} tells us how to obtain it. To derive a matrix notation as above, we order the occurring weights in $\beta$. We get $-w_{s+1} < -w_{s+2}$ and
		$$
			-w_{s+1}+w_{l+1} < \ldots < -w_{s+1}+w_{s+2} < -w_{s+2}+w_{l+1} < \ldots < -w_{s+2}+w_{s+3} < 0 < w_{s+1}-w_{s+2}.
		$$
		We choose a basis vector in each of these weight spaces and order them according to the above orders. Then we see that $R_{F_*}(M_\beta)$ is the set of tuples of matrices of the following form
		\begin{comment}
		$$
			\begin{blockarray}{cccccccccccccccccccc}
					& & A_1    &	    &		& A_s    &  		&   & A_{s+1} & & & A_{s+2}   & &   & A_{s+3} & &	& A_{l+1} & & \\
			\begin{block}{c(c(cc)c(cc)c(cc)c(cc)c(cc)c(cc)c)}
				1	& & *      & *      &		& *      & * 		&   & * & * &   & * & * 	&   & * & * &		& 1 & * & \\
				\vdots	& & \vdots & \vdots &		& \vdots & \vdots	&   & * & * &   & * & *		&   & * & * &		&   & * & \\
				r-1	& & \vdots & \vdots &		& \vdots & \vdots	&   & * & * &   & * & * 	&   & 1 & * &		&   & * & \\
				r	& & \vdots & \vdots &		& \vdots & \vdots	&   & * & * &   & 1 & * 	&   &   & * &		&   & * & \\
				r+1	& & \vdots & \vdots & ,\ldots,	& \vdots & \vdots	& , & * & * & , &   & *		& , &   & * &,\ldots,	&   & 1 & \\
				\vdots	& & \vdots & \vdots &		& \vdots & \vdots	&   & * & * &   &   & * 	&   &   & * &		&   &   & \\
				2r-1	& & \vdots & \vdots &		& \vdots & \vdots	&   & * & * &   &   & * 	&   &   & 1 &		&   &   & \\
				2r	& & \vdots & \vdots &		& \vdots & \vdots	&   & 1 & * &   &   & 1 	&   &   &   &		&   &   & \\
				2r+1	& & *      & *      &		& *      & * 		&   &   & 1 &   &   &  		&   &   &   &		&   &   & \\
			\end{block}
			\end{blockarray}\ .
		$$
		\end{comment}
		$$
			\begin{blockarray}{cccccccccccccc}
				& & A_1 & & A_s & & A_{s+1} & & A_{s+2} & & A_{s+3} & & A_{l+1} & \\
			\begin{block}{c(c(c)c(c)c(c)c(c)c(c)c(c)c)}
				1 & & \ray{*}{*} & & \ray{*}{*} & & \ray{*}{*} & & \ray{*}{*} & & \ray{*}{*} & & \ray{1}{*} & \\
				\vd & & \ray{\vd}{\vd} & & \ray{\vd}{\vd} & & \ray{\vd}{\vd} & & \ray{\vd}{\vd} & & \ray{\vd}{\vd} & & \ray{}{\vd} & \\
				r-2 & & \ray{\vd}{\vd} & & \ray{\vd}{\vd} & & \ray{\vd}{\vd} & & \ray{\vd}{\vd} & & \ray{*}{\vd} & & \ray{}{\vd} & \\
				r-1 & & \ray{\vd}{\vd} & & \ray{\vd}{\vd} & & \ray{\vd}{\vd} & & \ray{*}{\vd} & & \ray{1}{\vd} & & \ray{}{\vd} & \\
				r & & \ray{\vd}{\vd} & & \ray{\vd}{\vd} & & \ray{\vd}{\vd} & & \ray{1}{\vd} & & \ray{}{\vd} & & \ray{}{*} & \\
				r+1 & & \ray{\vd}{\vd} & ,\ldots, & \ray{\vd}{\vd} & , & \ray{\vd}{\vd} & , & \ray{}{\vd} & , & \ray{}{\vd} & ,\ldots, & \ray{}{1} & \\
				\vd & & \ray{\vd}{\vd} & & \ray{\vd}{\vd} & & \ray{\vd}{\vd} & & \ray{}{\vd} & & \ray{}{\vd} & & \ray{}{} & \\
				2r-2 & & \ray{\vd}{\vd} & & \ray{\vd}{\vd} & & \ray{\vd}{\vd} & & \ray{}{\vd} & & \ray{}{*} & & \ray{}{} & \\
				2r-1 & & \ray{\vd}{\vd} & & \ray{\vd}{\vd} & & \ray{*}{\vd} & & \ray{}{*} & & \ray{}{1} & & \ray{}{} & \\
				2r & & \ray{\vd}{\vd} & & \ray{\vd}{\vd} & & \ray{1}{*} & & \ray{}{1} & & \ray{}{} & & \ray{}{} & \\
				2r+1 & & \ray{*}{*} & & \ray{*}{*} & & \ray{}{1} & & \ray{}{} & & \ray{}{} & & \ray{}{} & \\
			\end{block}
			\end{blockarray}\ .
		$$
		A homogenous basis of the graded Lie algebra $\mathfrak{u}_{F_*}$ is given by the elements $x = \big(\big(\begin{smallmatrix}0 & 1 \\ 0 & 0 \end{smallmatrix}\big),0\big)$ and the elements $y_{i,j}=(0,E_{i,j})$ with $1 \leq i < j \leq 2r+1$; the matrix $E_{i,j}$ is defined as the matrix with an entry 1 in position $(i,j)$ and zeros elsewhere. So $[x,M_\beta]$ is, up to a sign, just right multiplication of each of the matrices of $M_\beta$ with the matrix $\big(\begin{smallmatrix}0 & 1 \\ 0 & 0 \end{smallmatrix}\big)$ while $[y_{i,j},M_\beta]$ is left multiplication with $E_{i,j}$. First we may use $y_{i,2s}$ to eliminate all entries above 1 in the first column of $A_{s+1}$. Then we eliminate the entry $(r,2)$ from the matrix $A_{s+2}$. Afterwards we employ $y_{i,j}$ with $j \neq 2s$ to annihilate all the entries above a 1 in every matrix, except for the first column of $A_{s+1}$ and the second column of $A_{s+2}$. We have used up all basis elements and we have arrived at the shape asserted in the Proposition.
	\end{proof}
	
	%\todo[inline]{Continue from here.}
	
	Next, we would like to determine the Poincar{\'e} series of $M^\theta(K(l+1),(2,2r+1))$. For this we need to analyze the covering dimension vectors of the second kind.
	The underlying vertices of the second type of covering dimension vector are
		$$
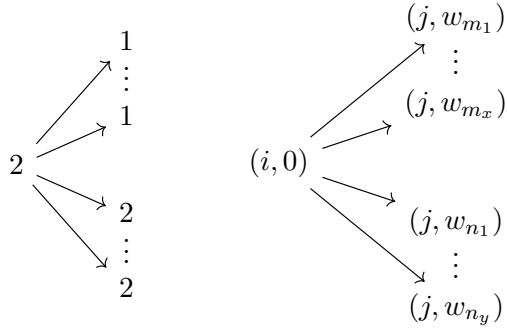
\begin{tikzcd}
					& 1 \\[-2.5em]
					& \vdots& \\[-2em]
					& 1	& \\[-1.5em]
				2 \arrow{ur} \arrow{uuur} \arrow{dr} \arrow{dddr}	&	& \\[-1.5em]
					& 2	& \\[-2.5em]
					& \vdots& \\[-2em]
					& 2	&
			\end{tikzcd}~~
			\begin{tikzcd}
					&  (j,w_{m_1}) \\[-2.5em]
					& \vdots& \\[-2em]
					&  (j,w_{m_x})	& \\[-1.5em]
				(i,0) \arrow{ur} \arrow{uuur} \arrow{dr} \arrow{dddr}	&	& \\[-1.5em]
					&  (j,w_{n_1})	& \\[-2.5em]
					& \vdots& \\[-2em]
					&  (j,w_{n_y})	&
			\end{tikzcd}
		$$
	for two sequences of indexes $m_1,\ldots,m_x$ and $n_1,\ldots,n_y$ in $\{1,\ldots,l+1\}$ such that 
	$$
		\sharp(\{m_1,\ldots,m_x\} \cup \{n_1,\ldots,n_y\})=x+y \leq l+1
	$$ 
	and $x+2y = 2r+1$; note that this forces $x$ to be odd. Denote this dimension vector by $\beta$. We compute
	$$
		\dim M^{\hat{\theta}}(\hat{K(l+1)},\beta) = 1 - \langle \beta,\beta \rangle = 3 - x,
	$$
	so for the moduli space associated with $\beta$ to be non-empty, we need $x \geq 3$ as well.
	By symmetry we may require $m_1 < \ldots < m_x$ and $n_1 < \ldots < n_y$. Let $1 \leq c_{1} < \ldots < c_t \leq l+1$  be the complementary indexes to $m_1,\ldots,m_x, n_1, \ldots, n_y$, where $t=l+1-x-y$. Let $[M]$ be an element of the fixed point component associated with $\beta$. Abbreviate $V = T_{[M]}M^\theta(K(l+1),d)$.
	\begin{align*}
		\dim V_\chi = &~ \delta(\chi,0) - \langle \beta,s_{-\chi}(\beta) \rangle \\
			= &~ \delta(\chi,0) - \left( 2\beta_{(i,-\chi)} - \sum_{\mu=1}^{x} \beta_{(j,w_{m_\mu}-\chi)}    - 2\sum_{\xi=1}^{t} \beta_{(j,w_{c_{\xi}}-\chi)}\right) \\
			=&-\delta(\chi,0) + \sum_{\mu=1}^{x} \beta_{(j,w_{m_\mu}-\chi)}    + 2\sum_{\xi=1}^{t} \beta_{(j,w_{c_{\xi}}-\chi)} \\
			=&-\delta(\chi,0) + \sum_{\mu=1}^{x}  \sum_{\mu'=1}^{x}  \delta(\chi, w_{m_\mu}-w_{m_{\mu'}}) +2\sum_{\mu=1}^{x}  \sum_{\nu=1}^{y}  \delta(\chi, w_{m_\mu}-w_{n_{\nu}}) \\ 
			& ~+ 2\sum_{\xi=1}^{t} \sum_{\mu=1}^{x}  \delta(\chi,w_{c_{\xi}}-w_{m_{\mu}})+ 4\sum_{\xi=1}^{t} \sum_{\nu=1}^{y}  \delta(\chi,w_{c_{\xi}}-w_{n_{\nu}})&
	\end{align*}
	We obtain the following value for the dimension of the attractor of $[M]$:
	\begin{align*}
		\dim \Att^+([M]) =  &\sum_{\chi > 0} \dim V_\chi \\
		= \binom{x}{2}
		  &  + 2 \sharp\{(\mu,\nu)\in\{1,\ldots,x\}\times\{1,\ldots,y\}\mid m_\mu<n_{\nu} \}\\
		  &  + 2 \sharp\{(\xi,\mu)\in\{1,\ldots,t\}\times\{1,\ldots,x\}\mid c_\xi< m_{\mu} \}\\
		  &  + 4 \sharp\{(\xi,\nu)\in\{1,\ldots,t\}\times\{1,\ldots,y\}\mid c_\xi< n_{\nu} \}
	\end{align*}
	We abbreviate $d_2^+(y,m_*,n_*) := \dim \Att^+([M])$. Note that $x = 2(r-y)+1$.
	
	Let us analyze the structure of the fixed point component $\smash{M^{\hat{\theta}}(\hat{K(l+1)},\beta)}$. It is easy to see that this moduli space is isomorphic to the moduli space of the $x$-subspace quiver ($x$ sources $i_1,\ldots,i_x$ and one sink $j$) with respect to the dimension vector
	$$
		\begin{tikzcd}
			1 \arrow{dr}{}& \\[-2em]
			\vdots & 2. \\[-1.5em]
			1 \arrow{ur}{}&
		\end{tikzcd}
	$$
	The stability condition is given by $\theta(e) = -e_j$. This subspace quiver moduli space can be seen to be isomorphic to the moduli space stable configurations $(p_1,\ldots,p_x)$ on $\P^1$ up to simultaneous $\PGL_2$-action; stability in this context means that at most $\lfloor x/2 \rfloor$ of $p_1,\ldots,p_x$ coincide. This is a classical moduli space. Its dimension is $x-3$ and its Betti numbers have been determined by Kirwan in \cite[Ex.\ 5.18]{Kirwan:84}. She shows that the dimension of the degree $2j$ cohomology group is
	$$
		b_j(x) := \dim H^{2j}(((\P^1)^x)^{\st}/\PGL_2;\Q) = 1 + \sum_{\nu = 1}^{\min(j,x-3-j)} \binom{x-1}{j}.
	$$
	
	Using the computations of the dimensions of the attracting sets, we arrive at the following formula for the Poincar\'{e} polynomial of the moduli space $M^\theta(K(l+1),d)$:
	
	\begin{prop}\label{prop:poincareEx}
		The Poincar\'{e} polynomial of the stable moduli space $X = M^\theta(K(l+1),d)$ of the $l+1$-Kronecker quiver for the dimension vector $d = (2,2r+1)$ is given by
		\begin{align*}
			P_X(t) =& \sum_{1 \leq m < n \leq l+1}\ \sum_{\substack{1 \leq m_1< \ldots <m_r \leq l+1 \\ m \notin \{m_1,\ldots,m_r\}}}\ \sum_{\substack{1 \leq n_1< \ldots <n_r \leq l+1 \\ n \notin \{n_1,\ldots,n_r\}}} t^{2d_1^+(m,m_*,n,n_*)} \\
			&+ \sum_{y=2r-l}^{r-1}\ \sum_{\substack{1 \leq m_1 < \ldots < m_{2(r-y)+1} \leq l+1 \\ 1 \leq n_1 < \ldots< n_y \leq l+1\\\text{disjoint}}} t^{2d_2^+(y,m_*,n_*)}\cdot \sum_{j=0}^{2(r-y)-2} b_j(2(r-y)+1)t^{2j}
		\end{align*}
	\end{prop}
	
	As a final application, we derive a cellular decomposition of a moduli space of the $3$-Kronecker quiver. Consider the dimension vector $d = (2,3)$. This means $l=r=2$. The dimension of the moduli space $M^\theta$ is $1 - \langle d,d \rangle_Q = 6$. As we have shown before, there are 13 stable $T$-fixed points which correspond to the following two types of covering quivers:
	\begin{comment}
	\begin{center}
		\begin{tikzpicture}[description/.style={fill=white,inner sep=2pt}]
			\matrix(m)[matrix of math nodes, row sep=.5em, column sep=2em, text height=1.5ex, text depth=0.25ex]
			{
					& 1 \\
				1	& \\
					& 1 \\
				1	& \\
					& 1 \\
			};
			\path[->, font=\scriptsize]
			(m-2-1) edge node[auto] {$m_1$} (m-1-2)
			(m-2-1) edge node[auto] {$m$} (m-3-2)
			(m-4-1) edge node[auto] {$n$} (m-3-2)
			(m-4-1) edge node[auto] {$n_1$} (m-5-2)
			;
		\end{tikzpicture}		\quad \quad
		\begin{tikzpicture}[description/.style={fill=white,inner sep=2pt}]
			\matrix(m)[matrix of math nodes, row sep=.5em, column sep=2em, text height=1.5ex, text depth=0.25ex]
			{
					& 1 \\
					~ &\\
				2	& 1 \\
					~&\\
					& 1 \\
			};
			\path[->, font=\scriptsize]
			(m-3-1) edge node[auto] {$a$} (m-1-2)
			(m-3-1) edge node[auto] {$b$} (m-3-2)
			(m-3-1) edge node[below left] {$c$} (m-5-2)
			;
		\end{tikzpicture}
	\end{center}
	where $m_1 \neq m_2 \neq m_3 \neq m_4$ (up to $S_2$-symmetry, there are 12 of them). 
	\end{comment}
	$$
		
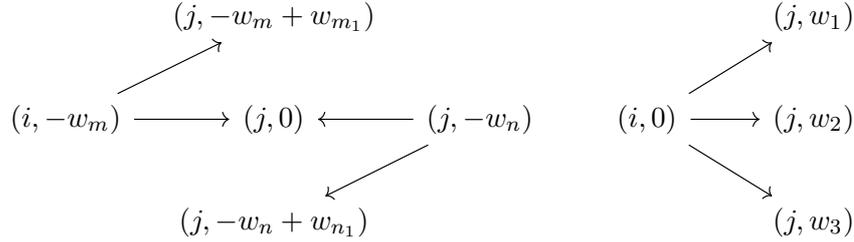
\begin{tikzcd}
					&[-1.5em](j,-w_m+w_{m_1}) 	&[-1.5em] \\
			(i,-w_m) \arrow{r}{} \arrow{ur}{}	& (j,0)		& (j,-w_n) \arrow{l}{} \arrow{dl}{} \\
					& (j,-w_n+w_{n_1})
		\end{tikzcd}
		\quad \quad
		\begin{tikzcd}
				& (j,w_1) \\
			(i,0) \arrow{r}{} \arrow{ur}{} \arrow{dr}{} & (j,w_2)\\
				& (j,w_3)
		\end{tikzcd}
	$$
	In the first type, we have $m_1,m,n,n_1 \in \{1,2,3\}$ such that $m_1 \neq m \neq n \neq n_1$. If we require $m < n$, then they are pairwise non-equivalent. There are 12 of those:
	$$
		1231, 2121, 1232, 2131, 3121, 3131, 2132, 3231, 2123, 3132, 3123, 3232.
	$$

	Let $\beta$ be a dimension vector of the first kind, given by the numbers $m_1,m,n,n_1$.
	We introduce the following symbols: we set $\delta_{x<y}=1$ if $x<y$ and $\delta_{x<y}=0$, otherwise. Then the formula for the dimension of the attractor of the fixed point $[M_\beta]$ reduces to
	\begin{align*}
		\dim \Att^+([M_\beta]) = -1 +\delta_{m<m_1}+\delta_{n<n_1}
		+ \delta_{m_2<m}+ \delta_{n_2<n}+ \delta_{m_2<m_1}+ \delta_{ \min(m_2,n)<\min(n_1,m) }\\+ \delta_{n_2 < n_1}+ \delta_{ \min(n_2,m) < \min(m_1,n)}
	\end{align*}
	For the unique fixed point $[M]$ which belongs to the covering dimension vector of the second kind, the attractor dimension formula reduces to $\dim \Att^+([M]) =   \binom{3}{2}  =3$.
	Proposition \ref{prop:poincareEx} yields the following expression for the Poincar\'{e} polynomial 
	$$
		P_{M^\theta(K(3),(2,3))}(t) = 1 + t^2 + 3t^4 + 3t^6 + 3t^8 + t^{10} + t^{12}.
	$$
	Finally we use Theorem \ref{t:att_explicit} to obtain a cell decomposition of the moduli space.
	\begin{prop}
		The moduli space $M^\theta(K(3),(2,3))$ has a cell decomposition which is given in the following table: %\enlargethispage{2cm}
		\begin{center}
		\begin{tabular}{cc|cc}
			$\beta$ & $\Att^+([M_\beta])$ & $\beta$ & $\Att^+([M_\beta])$ \\\hline
			$1231$ & 
			$\left( \begin{array}{cc} 0&0\\1&0\\0&1 \end{array}\right), 
			\left( \begin{array}{cc} 1&0\\0&0\\0&0 \end{array}\right), \left( \begin{array}{cc} 0&1\\0&0\\0&0\end{array}\right)$
			&
			$2312$ & 
			$\left( \begin{array}{cc} \color{purple}*&*\\1&0\\0&* \end{array}\right), 
			\left( \begin{array}{cc} 1&0\\0&0\\0&1 \end{array}\right), \left( \begin{array}{cc} 0&\color{purple}0\\0&1\\0&0 \end{array}\right)$ 
			\\
			$2121$ &
			$\left( \begin{array}{cc} 0&0\\1&0\\0&1 \end{array}\right), 
			\left( \begin{array}{cc} 1&0\\0&1\\0&0 \end{array}\right), \left( \begin{array}{cc} 0&*\\0&0\\0&0 \end{array}\right)$
			&
			$3231$ & 
			$\left( \begin{array}{cc} *&0\\ *&0\\0&1 \end{array}\right), \left( \begin{array}{cc} \color{purple}0&*\\1&\color{purple}*\\0&0 \end{array}\right), 
			\left( \begin{array}{cc} 1&\color{purple}0\\0&1\\ 0&0 \end{array}\right)$ 
			\\
			$1232$ &
			$\left( \begin{array}{cc} 0&*\\0&*\\1&0 \end{array}\right), 
			\left( \begin{array}{cc} 1&0\\0&1\\0&0 \end{array}\right), \left( \begin{array}{cc} 0&1\\0&0\\0&0 \end{array}\right)$
			 &
			$2123$ & 
			$\left( \begin{array}{cc} \color{purple}*&*\\\color{green}*&*\\1&\color{purple}0 \end{array}\right), 
			\left( \begin{array}{cc} 1&\color{purple}0\\0&\color{green}0\\0&1 \end{array}\right), 
			\left( \begin{array}{cc} 0&0\\0&1\\0&0 \end{array}\right)$
			\\
			$2131$ & 
			$\left( \begin{array}{cc} \color{purple}*&0\\1&0\\0&1 \end{array}\right), 
			\left( \begin{array}{cc} 1&0\\0&*\\0&0 \end{array}\right), \left( \begin{array}{cc} 0&\color{purple}0\\0&1\\0&0 \end{array}\right)$ 
			&
			$3132$ &
			$\left( \begin{array}{cc} \color{purple}0&*\\1&\color{purple}*\\0&* \end{array}\right), 
			\left( \begin{array}{cc} *&0\\0&0\\0&1 \end{array}\right), \left( \begin{array}{cc} 1&\color{purple}0\\0&1\\0&0 \end{array}\right)$ 
			\\
			$3121$ &$\left( \begin{array}{cc} \color{purple}*&0\\1&0\\ 0&1 \end{array}\right), 
			\left( \begin{array}{cc} *&\color{purple}0\\0&1\\0&0 \end{array}\right), 
			\left( \begin{array}{cc} 1&0\\0&0\\0&0 \end{array}\right)$
			&
			$3123$ &
			$\left( \begin{array}{cc} \color{purple}*&*\\\color{green}*&*\\1&0 \end{array}\right), 
			\left( \begin{array}{cc} *&\color{purple}0\\0&\color{green}0\\0&1 \end{array}\right), 
			\left( \begin{array}{cc} 1&0\\0&1\\0&0 \end{array}\right)$
			\\
			$3131$ &
			$\left( \begin{array}{cc} 0&0\\1&0\\0&1 \end{array}\right), 
			\left( \begin{array}{cc} *&*\\0&*\\0&0 \end{array}\right),
			\left( \begin{array}{cc} 1&0\\0&1\\0&0 \end{array}\right)$
			&
			$3232$ &
			$\left( \begin{array}{cc} *&*\\ *&*\\ *&* \end{array}\right), 
			\left( \begin{array}{cc} 0&0\\1&0\\0&1 \end{array}\right),
			\left( \begin{array}{cc} 1&0\\0&1\\0&0 \end{array}\right)$
			\\\hline
			$(2,1,1,1)$ &
			$\left( \begin{array}{cc} *&0\\ *&0\\0&1 \end{array}\right), 
			\left( \begin{array}{cc} \color{purple}0&\color{purple}*\\ 1&1\\0&0 \end{array}\right), 
			\left( \begin{array}{cc} 1&0\\0&0\\0&0 \end{array}\right)$ &\\
		\end{tabular}
		\end{center}
		The colors in the above table indicate where a choice of a complement has been made.
	\end{prop}

	\bibliographystyle{abbrv}

\end{document}